\documentclass[a4paper,11pt,centertags,reqno]{amsart}
\usepackage[english]{babel}
\usepackage{cite}
\usepackage{verbatim}
\usepackage{color}
\usepackage{hyperref}
\usepackage{extarrows}
\usepackage{amsmath,amsthm}
\usepackage{comment}

\usepackage[top=2.5cm, bottom=2.5cm, left=2.8cm,right=2.8cm]{geometry}

\definecolor{olive}{rgb}{0.3, 0.4, .1}			    
\definecolor{fore}{RGB}{249,242,215}			    
\definecolor{back}{RGB}{51,51,51}				    
\definecolor{title}{RGB}{255,0,90}				    
\definecolor{dgreen}{rgb}{0.,0.6,0.}			    
\definecolor{gold}{rgb}{1.,0.8,0.}				    
\definecolor{JungleGreen}{cmyk}{0.99,0,0.52,0}	    
\definecolor{BlueGreen}{cmyk}{0.85,0,0.33,0}	    
\definecolor{RawSienna}{cmyk}{0,0.72,1,0.45}	    
\definecolor{Magenta}{cmyk}{0,1,0,0}

\newtheorem{theo}{Theorem}[section]
\newtheorem{lemma}{Lemma}[section]
\newtheorem{prop}{Proposition}[section]
\newtheorem{cor}{Corollary}[section]

\theoremstyle{definition}

\newtheorem{rem}{Remark}[section]

\numberwithin{equation}{section}

\newcommand{\R}{\mathbb R}
\newcommand{\N}{\mathbb N}

\newcommand{\Q}{\mathcal Q}

\newcommand{\de}{\partial}
\newcommand{\eps}{\varepsilon}

\newcommand{\ds}{\displaystyle}
\DeclareMathOperator{\divergenza}{div}

\begin{document}
\title[Quasilinear operators with Robin boundary conditions]{Optimizing the first eigenvalue of some quasilinear operators with respect to the boundary conditions}
 \author{Francesco Della Pietra}
       \address{Universit\`a degli studi di Napoli Federico II, Dipartimento di Matematica e Applicazioni ``R. Caccioppoli'', Via Cintia, Monte S. Angelo - 80126 Napoli, Italia.}
       \email{f.dellapietra@unina.it}
 \author{Nunzia Gavitone}
       \address{Universit\`a degli studi di Napoli Federico II, Dipartimento di Matematica e Applicazioni ``R. Caccioppoli'', Via Cintia, Monte S. Angelo - 80126 Napoli, Italia.}
       \email{nunzia.gavitone@unina.it}
  \author{Hynek Kova\v{r}\'ik}
  \address{Universit\`a degli Studi di Brescia, DICATAM, Sezione di Matematica, Via Branze 38, 25123 Brescia, Italy}
  \email{hynek.kovarik@unibs.it}
 
\begin{abstract}
We consider a class of quasilinear operators on a bounded domain $\Omega\subset \R^n$  and address the question of optimizing the first eigenvalue with respect to the boundary conditions, which are of the Robin-type. We describe the optimizing boundary conditions and establish upper and lower bounds on the respective maximal and minimal eigenvalue. 
\end{abstract}

\maketitle

{\bf  Keywords:}
Robin boundary conditions, optimization problem, $p-$Laplacian\\

{\bf  MSC 2010: 35P15 (35P30, 35J60)}

\section{\bf Introduction}
The problem of optimizing first eigenvalues of certain differential operators is well-known from the literature mainly in connection with the so-called shape optimization. The latter means that one looks for a domain which minimizes (or maximizes) the first eigenvalue under some geometrical constraint, typically keeping the volume fixed. The answer in the case of the Laplace operator is that the minimum is achieved by a ball with the prescribed volume. This was proved in \cite{Fa} and \cite{Kra} for Dirichlet boundary conditions and in \cite{boss} for Robin boundary conditions. Various generalizations and improvements of these results appeared recently, see for example  \cite{bg1},\cite{bg2},\cite{daifu}, \cite{dan}, \cite{pota}, \cite{FNT15}, \cite{FK} and references therein. Another type of shape optimization, concerning domains with holes, was studied in \cite{him}, \cite{hen}, \cite{her}. 

\smallskip 

In this paper we analyze a different optimization problem; we keep a bounded domain $\Omega \subset \R^n$ fixed and vary the boundary conditions. More precisely, we consider the variational problem 
\begin{equation} \label{eq-varprob}
\inf_{u\in W^{1,p}(\Omega)}  \dfrac{\ds \int_\Omega |\nabla u|^p \, dx + \ds \int_{\de
     \Omega}\sigma\, |u|^p\, d\mathcal H^{n-1}}{\ds \int_\Omega |u|^p dx}\, , \qquad p> 1, 
\end{equation} 
and ask which function $\sigma: \partial\Omega\to [0,\infty[$ minimizes or maximizes \eqref{eq-varprob} under the condition 
\begin{equation} \label{m-cond}
\int_{\partial\Omega} \sigma \, d\mathcal H^{n-1} = m, 
\end{equation}
where $m$ is a positive constant. Under certain regularity conditions on $\Omega$ the infimum in \eqref{eq-varprob} is a minimum and the corresponding minimizer solves an eigenvalue equation for the $p-$Laplace operator with Robin-type boundary conditions, see Section \ref{sec-prelim} and equation \eqref{robinp} for details.

Our main results can be summarized as follows. We show that for sufficiently regular $\Omega$ the maximizing $\sigma$ always exists and is unique. In fact we provide its explicit construction, see Theorem \ref{thm3.1}. For examle if $\Omega$ is a ball then it turns out that the maximizing $\sigma$ is constant, see Remark \ref{rem-ball}. 

As for the minimum, we find that as soon as $n>1$ there is no $\sigma$ which minimizes \eqref{eq-varprob}  in the class of nonnegative functions satisfying \eqref{m-cond}. Moreover, if $p\leq n$, then the infimum of \eqref{eq-varprob} over $\sigma$ belonging to this class is zero, see Proposition \ref{inf0}. 
However, if $p>n$, then this infimum is positive, see Theorem \ref{thmp>n}, and is achieved in the class of  Dirac measures on $\partial\Omega$ of total mass $m$. In other words, it is achieved if $\sigma$ in \eqref{eq-varprob} is replaced by a Dirac measure concentrated at a point of the boundary, see Theorem \ref{carmin}.  The position of this point, which might not be unique, depends of course on $m$, but it is possible to describe its asymptotic behavior as $m\to\infty$. This is done in Proposition \ref{prop-xm}. 

Let us briefly outline the structure of the paper. In Section \ref{sec-prelim} we fix the necessary notation and provide some preliminary results which will be needed later. Section \ref{sec-sup} is devoted to the analysis of the $\sigma$ which maximizes \eqref{eq-varprob}. The minimum, or more precisely, the infimum is treated in Section \ref{sec-inf}. It is of course natural to ask how big or small the maximum and the minimum (or infimum) of \eqref{eq-varprob} are. Obviously, this depends on $m$ and on $\Omega$. In Section \ref{sec-estimates} we provide upper and lower bounds on these quantities and study their limits for $m\to 0$ and $m\to\infty$. 

\section{\bf Notation and preliminaries}
\label{sec-prelim}
Throughout the paper we will assume that $\Omega\subset\R^{n}$ is a bounded domain with $C^{1,\eps}$ regular boundary, and $1<p<+\infty$. We recall that under this assumption, the standard trace embedding theorem, see e.g.~\cite{adams}, assures that there exists a constant $C=C(\Omega,p,q)$ such that
\begin{equation}
\label{trace}
\|u\|_{L^{q}(\de\Omega)} \le C \|u\|_{W^{1,p}(\Omega)}, \text{ for } 
\left\{
\begin{array}{ll}
q= \frac{p(n-1)}{n-p}&\text{if }p<n,\\
q<+\infty &\text{if }p=n\\
q=+\infty &\text{if }p>n.
\end{array}
\right.
\end{equation}
Let us assume that $\sigma\in L^{1}(\de\Omega)$ is nonnegative, and consider the following Robin eigenvalue problem:
\begin{equation}
\label{mainvar}
\ell_{1}(\sigma,\Omega) = 
\inf_{\substack{u\in W^{1,p}(\Omega)\\ u\neq 0}} \Q[\sigma,u],
\end{equation}
where
\[
\Q [\sigma,u]= \dfrac{\ds \int_\Omega |\nabla u|^p dx + \ds \int_{\de
     \Omega}\sigma(x)|u|^pd\mathcal H^{n-1}}{\ds \int_\Omega |u|^p dx}
\]
if the right-hand side is finite, otherwise $\Q[\sigma,u]=+\infty$.

\smallskip

If $v\in W^{1,p}(\Omega)$ is a minimizer of \eqref{mainvar}, then it is a weak solution of the following Robin boundary value problem
 \begin{equation}
\label{robinp}
  \left\{
    \begin{array}{rcll}
    -\Delta_{p}v& =& \ell_{1}(\sigma,\Omega) |v|^{p-2}v
      &\text{in } 
      \Omega,\\[.3cm] 
      |\nabla v|^{p-2}\dfrac{\de v}{\de \nu} & = &-\sigma(x) |v|^{p-2}v      &\text{on } \de\Omega,
    \end{array}
  \right.
\end{equation} 
where   $\Delta_{p}v=\divergenza\left(|\nabla v|^{p-2}\nabla v\right)$ is the $p$-Laplace operator. By a weak solution to \eqref{robinp} we mean a function $v\in W^{1,p}(\Omega)$ such that $\sigma |v|^{p}\in L^{1}(\de\Omega)$ and
\begin{equation}
\label{defsol}
\int_{\Omega} |\nabla  v|^{p-2}\nabla v\cdot \nabla \varphi\,dx +\int_{\de\Omega}\sigma(x)|v|^{p-2}v\,\varphi\,d\mathcal H^{n-1}=
\ell_{1}(\sigma,\Omega)\int_{\Omega}|v|^{p-2}v\,\varphi\, dx
\end{equation}
for any test function $\varphi\in W^{1,p}(\Omega)\cap L^{\infty} (\de\Omega)$. The following result holds.
\begin{prop}
\label{lemma1}
Let $\sigma\in L^{1}(\de \Omega)$, and $\sigma \ge 0$. Then there exists a positive minimizer $u_{p}\in W^{1,p}(\Omega)$ of \eqref{mainvar}, which is a weak solution of  \eqref{robinp} in $\Omega$. Moreover, if $\sigma$ is positive on $\Gamma \subseteq \de\Omega$ such that the $(n-1)$-Hausdorff measure $\mathcal H^{n-1}(\Gamma)>0$, then $\ell_{1}(\sigma,\Omega)>0$. Finally, $\ell_{1}(\sigma,\Omega)$ is simple, that is $u_{p}$ is unique up to a multiplicative constant.
\end{prop}
\begin{proof}
Let $\varphi_{k}\in W^{1,p}(\Omega)$ be a minimizing sequence of \eqref{mainvar} such that $\|\varphi_{k}\|_{L^{p}(\Omega)}=1$. Then, being $\varphi_{k}$ bounded in $W^{1,p}(\Omega)$ and using the Rellich Theorem, there exists a subsequence, still denoted by $\varphi_{k}$, which weakly converges to a function $u_{p}\in W^{1,p}(\Omega)$, with $\|u_{p}\|_{L^{p}(\Omega)}=1$. The quoted trace inequality \eqref{trace} gives that, in particular, $\varphi_{k}$ converges almost everywhere on $\de\Omega$ to $u_{p}$. By Fatou's Lemma,
\[
\ell_{1}(\sigma, \Omega) =\lim_{k\rightarrow +\infty} \mathcal Q[\sigma,\varphi_{k}] \ge \mathcal Q[\sigma,u_{p}].
\]
Then $u_{p}$ is a minimum. To compete the proof of the first part of the Lemma we observe that $|u_{p}|$ is still a minimum, and then by the Harnack inequality $|u_{p}|>0$.

Finally, suppose by contradiction that there exists $\sigma_{1}>0$ on $\Gamma\subseteq \de \Omega$, with $\mathcal H^{n-1}(\Gamma)>0$ such that $\ell_{1}(\sigma_{1},\Omega)=0$. Then there exists $u_{\sigma_{1}}\in W^{1,p}(\Omega)$ such that $\|u_{\sigma_{1}}\|_{L^p(\Omega)}=1$ and
\[
0=\ell_{1}(\sigma_{1},\Omega)=\int_{\Omega}|\nabla u_{\sigma_{1}}|^{p} dx +\int_{\de\Omega} \sigma_{1}|u_{\sigma_{1}}|^{p}d\mathcal H^{n-1}.
\]
Hence $u_{\sigma_{1}}$ is constant in $\bar\Omega$, and $|u_{\sigma_1}|^{p}\int_{\de\Omega}\sigma_{1} d\mathcal H^{n-1}=0$. The hypothesis on $\sigma_{1}$ implies that $u_{\sigma_{1}}\equiv0$ in $\bar\Omega$, and this is impossible. Hence $\ell_{1}(\sigma,\Omega)>0$. The simplicity of $\ell_{1}(\sigma,\Omega)$ follows by standard arguments, see for example \cite{daifu} or \cite{bk,pota}.
\end{proof}

\noindent To conclude this section, we point out that any nonnegative eigenfunction must be a first eigenfunction. 

\begin{prop}
\label{pota}
Any nonnegative function $v\in W^{1,p}(\Omega)$, $v\not\equiv 0$, which satisfies
	\begin{equation}
	\label{pbdue}
 	\left\{
   	\begin{array}{rll}
      	-\Delta_p v &= \eta\, v^{p-1}
      	&\text{in } \Omega,\\[.2cm] 
      	|\nabla v|^{p-2}\, \nabla v\cdot \nu &=-\sigma(x)\, v^{p-1}
      	&\text{on } \de\Omega,
    	\end{array}
  \right.
	\end{equation}
in the weak sense, is a first eigenfunction of \eqref{pbdue}, that is $\eta=\ell_{1}(\sigma,\Omega)$ and $v=u_{p}$, where $u_{p}$ is given in Proposition \ref{lemma1}, up to a multiplicative constant.
\end{prop}

\begin{proof}
The proof follows line by line the argument given in \cite[Theorem 3.3]{pota}. 
\end{proof}

\smallskip

\subsection{Notation}
For a given $m> 0$, let us consider the set of functions 
\begin{equation}
\label{sigmam}
\Sigma_m(\de \Omega)=\left\{\sigma \in L^1(\de \Omega) \colon \sigma \ge 0, \, \int_{\de \Omega}\sigma \,d \mathcal{H}^{n-1}=m\right\}.
\end{equation}
Note that for every $\sigma \in \Sigma_{m}(\de\Omega)$ we have 
\begin{equation}
\label{uptriv}
 \ell_{1}(\sigma, \Omega) \le \min\left\{\Lambda_{1}^{D}(\Omega), \frac{m}{|\Omega|}\right\},
\end{equation}
where 
$$
\Lambda_{1}^{D}(\Omega) = \inf_{u\in W_0^{1,p}(\Omega)}  \dfrac{\ds \int_\Omega |\nabla u|^p dx}{\ds \int_\Omega |u|^p dx}
$$
denotes the first Dirichlet eigenvalue of $-\Delta_p $ in $\Omega$. Upper bound \eqref{uptriv} follows  by choosing as a test function in \eqref{mainvar} the first Dirichlet eigenfuntion of $-\Delta_p $ respectively a constant function. In view of \eqref{uptriv} we can thus define the quatities
\begin{align}
\Lambda(m,\Omega)& = \sup_{\sigma\in\Sigma_{m}(\de\Omega)} \ell_{1}(\sigma,\Omega), \label{sup} \\
& \nonumber \\
\lambda(m,\Omega)& = \inf_{\sigma\in\Sigma_{m}(\de\Omega)} \ell_{1}(\sigma,\Omega). \label{inf}
\end{align}
which are the main objects of our interest.

\section{\bf Optimization of $\ell_{1}(\sigma,\Omega)$ with respect to $\sigma$: the supremum}
\label{sec-sup}
The purpose of this section is to analyze the optimization problem \eqref{sup}. We start by showing that it is sufficient to study the supremum of $\ell_{1}(\sigma,\Omega)$  among the functions $\sigma \in \Sigma_m(\de \Omega)$ such that the corresponding minimiser $\hat{u}$ of $\mathcal{Q}[\sigma, u]$ is constant on the boundary of $\Omega$. 

\begin{prop}\label{ossmax}
Let $p>1$, $m >0$, $\hat\sigma\in \Sigma_{m}(\de\Omega)$. If $\hat u\in W^{1,p}(\Omega)$ is a function such that $\ell_{1}(\hat\sigma,\Omega)=\mathcal Q[\hat\sigma,\hat u]$ and $\hat u$ is constant on $\de \Omega$, then 
\[
\Lambda(m,\Omega)= \ell_1(\hat{\sigma},\Omega).
\]
\end{prop}
\begin{proof}Let us suppose that $\hat u$ is constant on $\de\Omega$. Then, for any $\sigma \in \Sigma_m(\de \Omega)$ we have:
\begin{gather*}
\begin{split}
\ell_{1}(\sigma,\Omega ) = & \min_{\substack{u\in W^{1,p}(\Omega)\\ u\neq 0}} \Q[\sigma,u] \le \Q[\sigma,\hat{u}] =\dfrac{\ds \int_\Omega |\nabla  \hat{u}|^p dx + \ds \int_{\de \Omega}\sigma(x)\hat{u}^p d\mathcal H^{n-1}}{\ds \int_\Omega \hat{u}^p dx}\\ 
= & \dfrac{\ds \int_\Omega |\nabla  \hat{u}|^p dx + \ds m\left.\hat{u}\right|_{\de \Omega}^{p} }{\ds \int_\Omega \hat{u}^p dx}= \Q[\hat{\sigma},\hat{u}] =  \ell_1(\hat{\sigma},\Omega).
\end{split}
\end{gather*}   
Hence
\[
\Lambda(m,\Omega) \le \ell_{1}(\hat \sigma,\Omega).
\]
Being $\hat \sigma\in \Sigma_{m}(\de\Omega)$, the above inequality is an equality and the proof is completed.
\end{proof}

\noindent In order to prove the existence and uniqueness of the maximising $\sigma$ we consider, for any fixed $\xi\in]0,\Lambda_{1}^{D}(\Omega)[$,  the following problem:
 \begin{equation}
 \label{aux}
  \left\{
    \begin{array}{rl}
    -\Delta_{p}v &=\left(\xi^{\frac{1}{p-1}}v+1\right)^{p-1}
      \quad \text{in }  \
      \Omega,\\[.2cm] 
	v &=0 \qquad\qquad\ \qquad \quad \text{on }\  \de\Omega.
    \end{array}
  \right.
\end{equation}
By \cite[Thms.~1\, \& 2]{diazsaa}, the condition $\xi< \Lambda_{1}^{D}(\Omega)$ guarantees that there exists a unique nonnegative solution $u_{\xi}$ of \eqref{aux}. The boundary regularity theory, see \cite[Thm.~1]{lieber}, shows that $u_{\xi}\in C^{1,\beta}(\overline{\Omega})$ for some $\beta>0$. 
Moreover, since $ \Delta_p u_{\xi} < 0$, by \cite[Thm.~5]{vazquez} we then conclude that $u_{\xi}$ is positive in $\Omega$ and 
$\frac {\de u_{\xi}}{\de\nu} < 0$ on $\de\Omega$. Hence
\begin{equation} \label{L-infty}
  \frac{u_{\xi_1}}{u_{\xi_2}} \in L^\infty(\Omega) \qquad \forall \ \xi_1, \xi_2 \in \ ]0,\Lambda_{1}^{D}(\Omega)[\, .
\end{equation}
Now let us define the function $F:\,  [0,\Lambda_{1}^{D}(\Omega)[\ \to \  [0,+\infty[$ by
\[
F(\xi)= \xi \int_{\Omega}\left(\xi^{\frac{1}{p-1}} u_{\xi}+1\right)^{p-1}dx.
\]

\smallskip

\begin{lemma}
\label{F}
The function $F$ is strictly increasing, and $F(\xi)\rightarrow +\infty$ as $\xi\rightarrow \Lambda_{1}^{D}(\Omega)$. 
\end{lemma}

\begin{proof} 
To simplify the notation, we write $\Lambda_{1}^{D}(\Omega)=\Lambda_{1}^{D}$. We split the proof in three steps.  

\medskip

{\bf Claim 1.} If $0\le \xi_{1}<\xi_{2}<\Lambda_{1}^{D}$, then  $u_{\xi_{1}}\le u_{\xi_{2}}$ in $\Omega$. We employ a variation of the argument used in \cite{diazsaa}, see also \cite{bk}. Let us define  
\[
	\varphi_{1}= \frac{(u_{\xi_{1}}^{p}-u_{\xi_{2}}^{p})^{+}}{u_{\xi_{1}}^{p-1}}\, , \qquad 
	\varphi_{2}= \frac{(u_{\xi_{1}}^{p}-u_{\xi_{2}}^{p})^{+}}{u_{\xi_{2}}^{p-1}} \, ,
\]
and
$$
\Omega_+ = \left\{ x\in \Omega\ : \ u_{\xi_1} > u_{\xi_2} \right\} . 
$$
In view of \eqref{L-infty} it is easy to see that the functions $\varphi_{1}$ and $\varphi_{2}$ are in $W_{0}^{1,p}(\Omega)$. Hence we may use $\varphi_{i}$ as test function for problem \eqref{aux} when $\xi=\xi_{i}$, $i=1,2$. By subtracting and integrating by parts we get
\begin{multline}
\int_{\Omega} \left(\frac{-\Delta_{p}u_{\xi_{1}}}{u_{\xi_{1}}^{p-1}}
+\frac{\Delta_{p}u_{\xi_{2}}}{u_{\xi_{2}}^{p-1}} \right) (u_{\xi_{1}}^{p}-u_{\xi_{2}}^{p})^{+}dx  = \\ 
= \int_{\Omega_+} \left[ |\nabla u_{\xi_1}|^p + |\nabla u_{\xi_2}|^p +(p-1) |\nabla u_{\xi_1}|^p \left(\frac{u_{\xi_2}}{u_{\xi_1}}\right)^p + (p-1) |\nabla u_{\xi_2}|^p \left(\frac{u_{\xi_1}}{u_{\xi_2}}\right)^p + \right. \\
 \left. -p\, \nabla u_{\xi_1}\cdot \nabla u_{\xi_2} \left(  \left(\frac{u_{\xi_2}}{u_{\xi_1}}\right)^{p-1} |\nabla u_{\xi_1}|^{p-2} + \left(\frac{u_{\xi_1}}{u_{\xi_2}}\right)^{p-1} |\nabla u_{\xi_2}|^{p-2}\right) \right]\, dx . \label{aux-2}
\end{multline}
We will show that the integrand on the right hand side of \eqref{aux-2} is nonnegative.  To this end consider the mapping $t \mapsto t^p-pt+p-1$ defined on $[0,\infty[$. By minimising with respect to $t$ we find that $t^p-pt+p-1 \geq 0$ for all $t\geq 0$. Therefore 
\[
\left(\frac{u_{\xi_1}}{u_{\xi_2}}\right)^p\, \frac{|\nabla u_{\xi_2}|^p}{|\nabla u_{\xi_1}|^p} +p-1\ \geq\ p\, \left(\frac{u_{\xi_1}}{u_{\xi_2}}\right)\, \frac{|\nabla u_{\xi_2}|}{|\nabla u_{\xi_1}|},
\]
which implies that 
\[
|\nabla u_{\xi_2}|^p + (p-1) |\nabla u_{\xi_1}|^p\,  \left(\frac{u_{\xi_2}}{u_{\xi_1}}\right)^p\ \geq\ p\, |\nabla u_{\xi_1}|^{p-1} |\nabla u_{\xi_2}| \left(\frac{u_{\xi_2}}{u_{\xi_1}}\right)^{p-1}
\]
provided $\nabla u_{\xi_1} \neq 0$. In the same way it follows that 
\[
|\nabla u_{\xi_1}|^p + (p-1) |\nabla u_{\xi_2}|^p\,  \left(\frac{u_{\xi_1}}{u_{\xi_2}}\right)^p\ \geq\ p\, |\nabla u_{\xi_2}|^{p-1} |\nabla u_{\xi_1}| \left(\frac{u_{\xi_1}}{u_{\xi_2}}\right)^{p-1}
\]
whenever $\nabla u_{\xi_2} \neq 0$. Since the positivity of \eqref{aux-2} is trivial on the set where $\nabla u_{\xi_1} \cdot\nabla u_{\xi_2} =0$, we conclude with
\begin{equation} \label{geq0}
\int_{\Omega} \left(\frac{-\Delta_{p}u_{\xi_{1}}}{u_{\xi_{1}}^{p-1}}
+\frac{\Delta_{p}u_{\xi_{2}}}{u_{\xi_{2}}^{p-1}} \right) (u_{\xi_{1}}^{p}-u_{\xi_{2}}^{p})^{+}dx \ \geq \ 0.
\end{equation}
On the other hand, by \eqref{aux} 
\begin{multline*}
\int_{\Omega} \left(\frac{\Delta_{p}u_{\xi_{2}}}{u_{\xi_{2}}^{p-1}} -\frac{\Delta_{p}u_{\xi_{1}}}{u_{\xi_{1}}^{p-1}}
\right) (u_{\xi_{1}}^{p}-u_{\xi_{2}}^{p})^{+}dx = \\= 
\int_{\Omega}\left[
\left(\xi_{1}^{\frac{1}{p-1}}+\frac{1}{u_{\xi_{1}}}\right)^{p-1}-
\left(\xi_{2}^{\frac{1}{p-1}}+\frac{1}{u_{\xi_{2}}}\right)^{p-1}
\right] (u_{\xi_{1}}^{p}-u_{\xi_{2}}^{p})^{+}dx \le 0.
\end{multline*}
This in combination with \eqref{geq0} shows that $u_{\xi_{1}}\le u_{\xi_{2}}$ in $\Omega$ and consequently $F(\xi_{1})<F(\xi_{2})$. 

\smallskip

\noindent In rest of the proof we show that $F(\xi)$ diverges as $\xi$ approaches $\Lambda_1^D$.

\medskip

{\bf Claim 2.} $\|u_{\xi}\|_{L^{\infty}(\Omega)} \to +\infty$ as $\xi \to \Lambda_{1}^D$. It follows from Claim 1 that the function $\xi \to \|u_{\xi}\|_{L^{\infty}(\Omega)} $ is nondecreasing on $]0, \Lambda_1^D[$. Hence we only have to show  that $\|u_{\xi}\|_{L^{\infty}(\Omega)}$ is unbounded. 
By contradiction, we suppose that  $\|u_{\xi}\|_{L^{\infty}(\Omega)}\le M $ for any $0\le \xi<\Lambda_{1}^{D}$ and some $M$. Then $u_{\xi}$ is uniformly bounded in $W_{0}^{1,p}(\Omega)$ and then it converges weakly to a function $\psi \in W_{0}^{1,p}(\Omega)\cap L^{\infty}(\Omega)$ which is a weak nonnegative solution of \eqref{aux} for $\xi=\Lambda_{1}^D$. Let us consider a constant $C>1$. We have:  
\begin{gather*}
\begin{split}
-\Delta_p(C\psi)=&-C^{p-1} \Delta_p \psi= C^{p-1}\left((\Lambda^D_1)^{\frac{1}{p-1}}\psi+1\right)^{p-1}=\left((\Lambda^D_1)^{\frac{1}{p-1}}C\psi+C\right)^{p-1}\\
&\ge \left[\left((\Lambda^D_1)^{\frac{1}{p-1}}+\delta_M(C)\right)C\psi+1\right]^{p-1}\text{ in }\mathcal D',
\end{split}
\end{gather*}
where $\delta_{M}(C)=\frac{C-1}{CM}$.
Hence $C\psi$ is a positive supersolution of \eqref{aux} for 
$$
\xi=\big((\Lambda^D_1)^{\frac{1}{p-1}}+\delta_M(C)\big) >(\Lambda^D_1)^{\frac{1}{p-1}}.
$$ 
Being $v=0$ a subsolution, then for such $\xi$ there exists a nonnegative weak solution $w\in W_{0}^{1,p}(\Omega)\cap L^{\infty}(\Omega)$. This contradicts the necessary condition for the existence of a solution of \eqref{aux} formlated in \cite[Thm.2]{diazsaa}, and Claim 2 is proved.

\medskip

{\bf Claim 3}. When $\xi\rightarrow \Lambda_{1}^{D}$, then
\begin{equation} \label{eq-enough}
\frac{u_{\xi}}{\|u_{\xi}\|_{\infty}} \ \rightarrow \ v  \in W_{0}^{1,p}(\Omega) \cap L^{\infty}(\Omega),
\end{equation}
where $v$ is the first positive Dirichlet eigenfunction of $\Delta_{p}$ such that $\| v\|_\infty=1$.  To prove the above claim, we point out that the function 
$$
v_{\xi}:=\frac{u_{\xi}}{\|u_{\xi}\|_{\infty}} 
$$ 
satisfies
 \begin{equation*}
  \left\{
    \begin{array}{rcll}
    -\Delta_{p}v_{\xi} &=& \left(\xi^{\frac{1}{p-1}}v_{\xi}+\frac{1}{\|u_{\xi}\|_{\infty}}\right)^{p-1}
      &\text{in } 
      \Omega,\\[.3cm] 
	v_{\xi}& =& 0 &\text{on } \de\Omega.
    \end{array}
  \right.
\end{equation*}
and $\max_{\Omega} v_{\xi}=1$, hence $v_{\xi}$ are bounded in $W_{0}^{1,p}(\Omega)$ and then weakly converge in $W_{0}^{1,p}(\Omega)$, as $\xi \to \Lambda_{1}^{D}$, to a function $v\ge 0$ such that
\begin{equation}
\label{limitv}
  \left\{
    \begin{array}{rcll}
    -\Delta_{p}v&=& \Lambda_{1}^{D}v^{p-1}
      &\text{in } \Omega,\\[.3cm] 
	v&=& 0 &\text{on } \de\Omega.
    \end{array}
  \right.
\end{equation}
Moreover, since $v_{\xi}$ are uniformly bounded in $L^{\infty}(\Omega)$, they are also uniformly bounded in $C^{0,\alpha}(\bar\Omega)$ for some $\alpha\in]0,1[$, see e.g~\cite[Thm.~1]{lieber} or \cite{ladyz}. Hence $v_{\xi}$ converges to $v$ uniformly. This ensures that $\max_{\Omega} v=1$, and then $v\not\equiv 0$. Actually, $v$ is a first Dirichlet eigenfunction of $\Delta_{p}$. This proves \eqref{eq-enough}  which in turn implies that
\begin{equation}
\lim_{\xi\rightarrow \lambda_{1}^{D}}\frac{F(\xi)}{\|u_{\xi}\|^{p-1}_{\infty}}= \Lambda_{1}^{D}\int_{\Omega} v^{p-1}\,dx>0. 
\end{equation}
Hence in view of Claim 2 we have $F(\xi)\rightarrow +\infty$ as $\xi\rightarrow \Lambda_{1}^{D}(\Omega)$ and the proof of the Lemma is complete.
\end{proof}

\noindent Lemma \ref{F} allows us to define the function $\xi: \, ]0,\infty[ \ \to\  ]0,\Lambda_1^D(\Omega)[$ by 
\[
\xi(m):= F^{-1}(m).
\]
For each $m>0$ there exists a unique function $u_{\xi(m)}$ which solves problem \eqref{aux} for $\xi=\xi(m)$. 

\smallskip

\noindent We are now in position to give an explicit formula for $\sigma$ which maximises $\ell_1(\sigma,\Omega)$. 

\begin{theo}
\label{thm3.1}
The supremum $\Lambda(m,\Omega)=\ds\sup_{\sigma\in\Sigma_{m}(\de\Omega)} \ell_{1}(\sigma,\Omega)$ is attained for any $m>0$, and satisfies
\[
\Lambda(m,\Omega)=\ell_{1}(\sigma_{m},\Omega)= \xi(m),
\]
where 
\[
\sigma_{m}=- \xi(m)\, |\nabla u_{\xi(m)}|^{p-2}\ \frac{\de u_{\xi(m)}}{\de \nu}\, ,
\]
and $u_{\xi(m)}$ is the unique solution of \eqref{aux} with $\xi=\xi(m)$. Moreover, the maximiser $\sigma_{m}$ is unique.
\end{theo}

\begin{proof}
We first prove that $\sigma_m \in \Sigma_m(\de \Omega)$. Indeed  by the divergence theorem contained in \cite{anz} and by the definitions of $ \sigma_m, F$ and $\xi(m)$ , we have
\begin{gather*}
\begin{split}
\int_{\de \Omega} \sigma_m \,d \mathcal{H}^{n-1}=& - \xi(m)\int_{\de \Omega}  |\nabla u_{\xi(m)}|^{p-2} \frac{\de u_{\xi(m)}}{\de \nu} \,d \mathcal{H}^{n-1}=  \xi(m)\int_{ \Omega}- \Delta_pu_{\xi(m)} \, dx \\ &= \xi(m)\int_{ \Omega} \left(\xi^{\frac{1}{p-1}}u_{\xi(m)}+1\right)^{p-1} \, dx=F(\xi(m))=m.
\end{split}
\end{gather*}
We claim that $u_m=\xi(m)^{\frac{1}{p-1}}u_{\xi(m)}+1$ is a solution to the problem \eqref{robinp} with $\sigma=\sigma_{m}$. Indeed
\begin{gather*}
\begin{split}
- \Delta_pu_m=-\xi(m) \Delta_pu_{\xi(m)}=\xi(m) \left(\xi^{\frac{1}{p-1}}u_{\xi(m)}+1\right)^{p-1}=\xi(m) u_m^{p-1}.
\end{split}
\end{gather*}
As regards the boundary condition, we have
\begin{gather*}
\begin{split}
|\nabla u_m|^{p-2}\, \frac{\de u_m}{\de \nu}=\xi(m)|\nabla u_{\xi(m)}|^{p-2}\, \frac{\de u_{\xi(m)}}{\de \nu}=-\sigma_m.
\end{split}
\end{gather*}
Since $u_m=1$ on $\partial\Omega$, we have shown that $u_m$ is a solution to the problem \eqref{robinp} with $\sigma=\sigma_{m}$. Moreover, by Proposition \ref{ossmax} it follows that $\ell_{1}(\sigma_{m},\Omega)=\Lambda(m,\Omega)$. On the other hand, being $u_{m}>0$ in $\Omega$, Proposition \ref{pota} implies that $\ell_{1}(\sigma_{m},\Omega)=\xi(m)$. 

\smallskip

\noindent To conclude the proof it remains to show the uniqueness of $\sigma_{m}$.  Let $\bar \sigma\in \Sigma_{m}(\de\Omega)$ another maximiser. Reasoning as in Proposition \ref{ossmax}, and recalling that $u_{m}=1$ on $\de\Omega$, then 
\[
\ell_{1}(\sigma_{m},\Omega)= 
\ell_{1}(\bar\sigma,\Omega)\le \mathcal Q[\bar\sigma,u_{m}]=
\mathcal Q[\sigma_{m},u_{m}]=\ell_{1}(\sigma_{m},\Omega).
\]
Then $u_{m}$ satisfies \eqref{robinp} with $\sigma=\bar\sigma$. Hence
$\bar\sigma= -|\nabla u_m|^{p-2}\frac{\de u_m}{\de \nu}=\sigma_{m}$ almost everywhere on $\de \Omega$.
\end{proof}

\noindent Let us notice that the problem of the maximizing $\sigma$ in the linear case $p=2$ was treated already in \cite{kov}. 

\begin{rem} \label{rem-ball}
If $\Omega$ is a ball, then the unique positive solution of \eqref{aux} is a radial function. Hence in this case Theorem \ref{thm3.1} implies that the maximizing $\sigma$ is constant;
$$
\sigma_m =\frac{m}{|\partial\Omega|}\, .
$$
\end{rem}

\section{\bf Optimization of $\ell_{1}(\sigma,\Omega)$ with respect to $\sigma$: the infimum}
\label{sec-inf}
The aim of this section is to describe the behavior of the infimum of $\ell_{1}(\sigma,\Omega)$ when $\sigma\ge 0$ has a fixed $L^{1}-$norm. 
Our purpose consists in the analysis of the problem \eqref{inf} for a given $m>0$.
We will prove that $\lambda(m,\Omega)$ is never achieved, unless $n=1$. 
Moreover $\lambda(m,\Omega)$ is positive if and only if $p>n$. 

\subsection{The case $p\le n$}
\begin{prop}
\label{inf0}
Let $1<p\le n$, $n\ge 2$, and $m>0$. Then
\[
\lambda(m,\Omega)=0
\]
and the infimum is not achieved.
\end{prop}
\begin{proof}
Let us denote by $B_{r}(x)$ the ball centered at $x$ with radius $r>0$. For $x_{0}\in \de\Omega$ fixed, and for any $j\in\mathbb N$ let 
\[
\sigma_{j}(x)=
\left\{
\begin{array}{ll}
\alpha_{j} &\text{if }x\in B_{2^{-j}}(x_{0})\cap \de \Omega,\\[.2cm]
0&\text{elsewhere},
\end{array}
\right.
\]
where $\alpha_{j}>0$ is a number such that $\|\sigma_{j}\|_{L^{1}(\de\Omega)}=m$. 

\smallskip

\noindent If $p<n$, let
\[
u_{j}(x)=\left\{
\begin{array}{ll}
j|x-x_{0}| &\text{in } B_{\frac1j}(x_{0})\cap\Omega,\\[.2cm]
1 &\text{in } \Omega\setminus B_{\frac1j}(x_{0}). 
\end{array}
\right.
\]
Then
\[
0\le \mathcal{Q}[\sigma_{j},u_{j}] \le\frac {j^{p}\left|B_{\frac1j}(x_{0})\right| + j^{p}\,2^{-jp}\,m}{|\Omega|}\to 0\text{ as }j\to \infty.
\]
If $p=n$, let
\[
u_{j}(x)=\left\{
\begin{array}{ll}
\dfrac{-\log j}{\log(|x-x_{0}|)} &\text{in } B_{\frac1j}(x_{0})\cap\Omega,\\[.4cm]
1 &\text{in } \Omega\setminus B_{\frac1j}(x_{0}). 
\end{array}
\right.
\]
As before, a direct computation shows that $\ds\lim_{j\rightarrow +\infty} \mathcal Q[\sigma_{j},u_{j}]=0$. 
Finally, Lemma \ref{lemma1} assures that the infimum is not attained.
\end{proof}

\subsection{The case $p>n$} 

\subsubsection{\bf Positivity of the infimum} The substantial difference to the case  $p\leq n$ is that now $\lambda(m,\Omega)$ is positive.  

\begin{theo}
\label{thmp>n}
If $p>n$, then 
\[
\lambda(m,\Omega)>0.
\]
Moreover, if $n>1$, then $\lambda(m,\Omega)$ is not achieved.
\end{theo}
\begin{proof}
Let $p >n$. We first show, arguing by contradiction, that the infimum is positive. Let us suppose that $\sigma_{k}\in \Sigma_{m}(\de\Omega)$ and $u_{k}\in W^{1,p}(\Omega)$ are such that
\[
\ell_{1}(\sigma_{k},\Omega)=\mathcal Q[\sigma_{k},u_{k}]\rightarrow 0\quad \text{as }k\to+\infty.
\]
Moreover, we assume that $u_{k}\ge 0$ and $\|u_{k}\|_{p}=1$. Hence, we have that
\begin{equation}
\label{propeq1}
\int_{\Omega} |\nabla u_{k}|^{p}dx\rightarrow 0,\qquad \int_{\de \Omega}\sigma_{k}\, u_{k}^{p}\, d\mathcal H^{n-1}\rightarrow 0.
\end{equation}
Together with the Morrey inequality, see e.g.~\cite[Thm.~5.6.5]{evans}, and the condition $\|u_{k}\|_{p}=1$, we have that, up to a subsequence, $u_{k}$ converges in $C^{0,\alpha}(\bar\Omega)$ to a constant $C>0$. Hence, passing to the limit in \eqref{propeq1}, and recalling that $\int_{\de\Omega}\sigma_{k}=m$, we have
\[
0=\lim_{k\to \infty}\left(\int_{\de\Omega}\sigma_{k}(u_{k}^{p}-C^{p})d\mathcal H^{n-1}+C^{p}\int_{\de\Omega}\sigma_{k}d\mathcal H^{n-1}\right)=C^{p}m
\]
that gives that $C=0$. This contradicts the condition $\|u_{k}\|_{p}=1$, and then the infimum $\lambda(m,\Omega)$ is positive.

Now we prove that if $p>n>1$, the infimum is not achieved. If $\lambda(m,\Omega)$ were a minimum, then $\bar{\sigma} \in \Sigma_m(\de \Omega)$ exists such that 
\[
\lambda(m,\bar{\sigma})= \mathcal{Q}[\bar\sigma,\bar u],
\]
where $\bar u\in W^{1,p}(\Omega)$, $\bar{u} \ge 0$ and $\bar{u}$ is not constant on $\de \Omega$, see Proposition \ref{ossmax}. Being $\bar u\in C^{0,\alpha}(\bar \Omega)$, we take $x_{0}\in \de \Omega$ such that $\bar u(x_{0})=\min_{ \de\Omega} \bar{u}$.  Then 
\begin{equation}
\label{positive}
\int_{\de \Omega} \bar{\sigma} \, |\bar{u}|^p\,d\mathcal H^{n-1}-m \,\bar u(x_{0})^p >0 .
\end{equation}
Let  $\sigma_j \in \Sigma_m(\de \Omega)$ such that  
\begin{equation} \label{convsigma}
\int_{\de \Omega} \sigma_j \, |\bar{u}|^p \,d\mathcal H^{n-1} \to m \,\bar u(x_{0})^p.
\end{equation}
For example, we can choose 
\begin{equation}  \label{convsigma-2}
\sigma_{j}(x)=
\left\{
\begin{array}{ll}
\alpha_{j} &\text{if }x\in B_{\frac1j}(x_{0})\cap \de \Omega,\\[.2cm]
0&\text{elsewhere},
\end{array}
\right.
\end{equation}
where $\alpha_{j}>0$ is a number such that $\|\sigma_{j}\|_{L^{1}(\de\Omega)}=m$. The continuity of $\bar u$ up to the boundary of $\Omega$ guarantees that \eqref{convsigma} holds. Hence, recalling \eqref{positive} there exists $k \in \mathbb{N}$ which satisfies
\[
\int_{\de \Omega} \sigma_k |\bar{u}|^p\,d\mathcal H^{n-1}<\int_{\de \Omega} \bar{\sigma} |\bar{u}|^p\,d\mathcal H^{n-1}.
\]
This implies that $\mathcal{Q}[\sigma_k, \bar{u}]<\lambda(m,\Omega)$ which is a contradiction.
\end{proof}

\subsubsection{\bf The relaxed problem and the concentration effect} In view of Theorem \ref{thmp>n} it is natural,  for $p>n$, to consider the relaxed variational problem
\begin{equation} \label{relax}
\ell_{1}(\mu,\Omega)=   \inf_{u\in W^{1,p}(\Omega)}\frac{\displaystyle\int_{\Omega}|\nabla u|^{p} dx+\int_{\de\Omega}|u| ^{p}\, d\mu}{\displaystyle\int_{\Omega}|u|^{p}dx}, \qquad \mu\in \mathcal M(m), 
\end{equation}
where 
$$
\mathcal M(m) : = \left\{ \text{set of Radon measures on}\  \de \Omega \ \text{ such that}\  \mu(\de \Omega)=m \right\} .
$$ 
Moreover, we introduce the subset of $\mathcal M(m)$ consisting of Dirac measures concentrated at a boundary point of $\Omega$. 
$$
\mathcal D(m) : = \left\{ \mu\in\mathcal M(m) \, : \ \exists\, x \in\partial\Omega \, :  \, \int_{\de\Omega}|u| ^{p}d\mu= m\, |u(x)|^p\ \ \  \forall \ u\in W^{1,p}(\Omega)  \right\} .
$$ 

\medskip

\noindent Armed with this notation we can show that $\lambda(m,\Omega)$ is equal to the minimum of $\ell_1(\mu,\Omega)$ on $\mathcal D(m)$.  

\begin{theo}
\label{carmin}
Let $p>n$. Then for any $m>0$ there exists $x_m\in\de\Omega$ such that 
\begin{equation} \label{eq-dirac}
\lambda(m,\Omega) = \ell_1(\mu_m,\Omega) = \inf_{\mu\in \mathcal D(m)}\ell_{1}(\mu,\Omega), 
\end{equation}
where $\mu_m\in \mathcal D(m)$ is the Dirac measure concentrated at  $x_m$.
\end{theo}

\begin{proof}
Let $\mu\in \mathcal D(m)$ and let $u_\mu$ be the corresponding minimizer for $\ell_1(\mu,\Omega)$. 
In view of the proof of Theorem \ref{thmp>n}, see equations \eqref{convsigma} and \eqref{convsigma-2}, there exists a sequence $\sigma_k\in \Sigma_m(\partial\Omega)$ such that 
$$
\int_{\de \Omega} \sigma_k \, |u_\mu|^p \,d\mathcal H^{n-1} \ \to\ \int_{\de \Omega}  |u_\mu|^p \, d\mu\, ,
$$
as $k\to \infty$. This shows that 
\begin{equation} \label{les-dir}
\lambda(m,\Omega)  \ \leq \  \ell_1(\mu,\Omega)  \qquad \forall\, \mu\in \mathcal D(m),
\end{equation}
and therefore 
\begin{equation} \label{les-dir-inf}
\lambda(m,\Omega) \, \leq \,  \inf_{\mu\in \mathcal D(m)}\ell_{1}(\mu,\Omega) .
\end{equation}
To prove the opposite inequality let $\sigma_{j}\in \Sigma_{m}(\de\Omega)$ be a minimizing sequence for $\lambda(m,\Omega)$. In other words, $\ell_{1}(\sigma_{j},\Omega)\to \lambda(m,\Omega)$. We denote by $u_{j,m}\in W^{1,p}(\Omega)$ the nonnegative functions such that $\|u_{j,m}\|_{L^{p}}=1$ and $\ell_{1}(\sigma_{j},\Omega)=\mathcal Q[\sigma_{j},u_{j,m}]$. Then 
\begin{equation}
\label{infj}
\int_{\Omega}|\nabla u_{j,m}|^{p} \,dx +\int_{\de \Omega}\sigma_j\, u_j^p \, d \mathcal{H}^{n-1} \to \lambda(m,\Omega).
\end{equation}
Being $p>n$ and $\de \Omega$ of class $C^{1,\eps}$, equation \eqref{infj} and the Morrey inequality, assure that $u_{j,m}$ is a bounded sequence in $C^{0,\alpha}(\bar\Omega)$. Hence, up to a subsequence, $u_{j,m}$ converges uniformly to some nonnegative $\bar u_m\in C^{0,\alpha}(\bar\Omega)$. On the other hand, $\sigma_{j}$ is uniformly bounded in $L^1(\de\Omega)$. Hence it contains a subsequence,  which we still denote by $\sigma_j$, converging weakly in the sense of measures to some $\mu\in \mathcal M(m) $. Then, $\mu(\de\Omega)=m$ and
\[
\lambda(m,\Omega)=
\int_{\Omega}|\nabla \bar u_m|^{p} \,dx +\int_{\de \Omega} \bar u_m^p \,d \mu.
\]
Now, let $x_m\in\de\Omega$ be such that $\bar u_m(x_m)=\min_{\de\Omega}\bar u_m$, and consider $\mu_m=m\delta_{x_m}$, where $\delta_{x_m}$ is the Dirac measure of unit mass concentrated at $x_m$. Then
\begin{align} 
\lambda(m,\Omega) & =
\int_{\Omega}|\nabla \bar u_m|^{p} \,dx +\int_{\de \Omega} \bar u_m^p \,d \mu
\ \ge 
\int_{\Omega}|\nabla \bar u_m|^{p} \,dx +m\, \bar u_m(x_m)^p \nonumber \\
&=  \int_{\Omega}|\nabla \bar u_m|^{p} \,dx +\int_{\de \Omega} \bar u_m^p \,d \mu_m 
 \ge \ell_1(\mu_m,\Omega). \label{eq-major}
\end{align} 
This in combination with \eqref{les-dir-inf} completes the proof. 
\end{proof}

\smallskip

\noindent The point of concentration $x_m$ introduced in Theorem \ref{carmin} need not be unique, since the domain $\Omega$ might possess some rotational symmetries. Indeed, in case of a ball it is obvious that $\ell_1(\mu,\Omega)=\ell_1(\nu,\Omega)$ for all $\mu, \nu \in \mathcal D(m)$. In general, the position of $x_m$ depends in a complicated way on $m$ and $\Omega$.  

\smallskip

\noindent However, it is possible to characterize the behavior of convergent subsequences of $x_m$ in the limit  $m\to \infty$. Note that the existence of at least one convergent subsequence is guaranteed by the Bolzano-Weierstrass Theorem. 
It turns out that the limiting behavior os these sequences is related to the following eigenvalue problem: 
\begin{equation}
\label{dirpunto}
 \lambda_{1}(x;\Omega) := \inf\left\{\frac{\|\nabla u\|_{L^p(\Omega)}^p}{\| u\|_{L^p(\Omega)}^p};\; \, u\in W^{1,p}(\Omega), \ u(x)=0\right\}, \quad x\in\de\Omega .
\end{equation}
By Lemma \ref{cont}, see Section \ref{sec-aux}, the function $\lambda_1(\cdot\, ;\Omega)$ is continuous and therefore admits a minimum on $\de\Omega$:
\begin{equation}
\label{lambda0}
\lambda_1(\Omega):=\min\left\{\lambda_{1}(x;\Omega);\; \, x\in\de\Omega\right\}.
\end{equation}
We have

\begin{prop} \label{prop-xm}
Any convergent subsequence of $x_m$ tends to a point of minimum of $\lambda_1(\cdot\, ;\Omega)$ as  $m\to \infty$. 
\end{prop}

\begin{proof}
\noindent Let $\bar u_m$ and $x_m$ be as in the proof of Theorem \ref{carmin} so that $\|\bar u_m\|_p=1$ and 
$$
\int_{\Omega}|\nabla \bar u_m|^{p} \,dx +m\, \bar u_m(x_m)^p = \lambda(m,\Omega).
$$
By definition of $\lambda_1(x; \Omega)$,  and equation \eqref{les-dir} it follows that 
\begin{equation} \label{lambda-upperb1}
\lambda(m,\Omega) \ \leq \ \lambda_1(x; \Omega) \qquad \forall\ m, \ \  \forall\, x\in\partial\Omega. 
\end{equation}
Hence  
\begin{equation} \label{m-0}
\bar u_m(x_m) \to 0, \qquad m\to\infty.
\end{equation}
Now let $\bar u\in W^{1,p}(\Omega)$ be a weak limit of (a weakly convergent subsequence of) $\bar u_m$. Then $\|\bar u\|_p=1$ by the Rellich-Kondrachov Theorem, see e.g. \cite[Thm.~8.9]{LL}. 
Next consider a convergent subsequence of $x_m$. Let $\bar x\in\partial\Omega$ be its limit. Then by \eqref{m-0} we have $\bar u (\bar x)=0$. Hence $\bar u$ is an admissible test function for $\lambda_1(\bar x; \Omega)$ and from \eqref{eq-major} and the weak lower-semicontinuity of $\int_\Omega |\nabla u|^p$ we infer that
\begin{align*}
\liminf_{m\to\infty} \lambda(m,\Omega) \ & \geq\  \liminf_{m\to\infty} \int_\Omega |\nabla \bar u_m|^{p} \,dx \ \geq \ \int_\Omega |\nabla \bar u|^{p} \,dx \\
& \ \geq \lambda_1(\bar x; \Omega) \geq \lambda_1(\Omega)\, .
\end{align*}
In view of \eqref{lambda-upperb1} it thus follows that $\lambda_1(\bar x; \Omega)=\lambda_1(\Omega)$.
\end{proof}

\begin{rem}
As in the linear case $p=2$, see  \cite[Prop. 3.5]{kov}, it can be shown that the in the one-dimensional case, when $\Omega=]a,b[$, the minimum of $\ell_1(m,\Omega)$ is achieved and that 
\[
\lambda(m,\Omega) = \ell_{1}(\sigma_{a},\Omega)=\ell_{1}(\sigma_{b},\Omega)
\]
where $\sigma_{a},\sigma_{b}\in \Sigma_{m}(\{a,b\})$ are such that $\sigma_{a}(a)=m$, $\sigma_{a}(b)=0$, and  $\sigma_{a}(b)=0$, $\sigma_{a}(b)=m$.
\end{rem}

\subsection{Auxiliary results concerning $\lambda_1(x, \Omega)$}
\label{sec-aux}

\begin{lemma} \label{prop-min}
Let $p>n$ and let $x\in\Omega$. Then $\lambda_1(x;\Omega)$ defined by \eqref{dirpunto} is a minimum. 
\end{lemma}

\begin{proof}
Consider a minimising sequence $u_k\in W^{1,p}(\Omega)$ such that 
$$
\| u_k\|_{L^p(\Omega)} =1, \qquad u_k(x)=0 \qquad \forall \ k\in\N. 
$$
Clearly $u_k$ is bounded in $W^{1,p}(\Omega)$ and therefore contains a subsequence which converges weakly to a function $u\in W^{1,p}(\Omega)$. The Morrey inequality then implies that $\| u\|_{L^p(\Omega)} =1$ and $u(x)=0$. Hence in view of the weak lower semi-continuity of $\|\nabla u\|_{L^p(\Omega)}^p$ we have 
$$
\lambda_{1}(x;\Omega) = \lim_{k\to\infty}  \|\nabla u_k\|_{L^p(\Omega)}^p \ \geq \ \|\nabla u\|_{L^p(\Omega)}^p \ \geq \ \lambda_{1}(x;\Omega)\, .
$$
This shows that $u$ is a minimiser for the problem \eqref{dirpunto}. 
\end{proof}

\begin{lemma}
\label{cont}
Let $p>n$. Then there exists a constant $C(n,p,\Omega)$ such that 
\begin{equation} \label{hoelder-cont}
|\lambda_1(x,\Omega) -\lambda_1(y,\Omega) | \ \leq \ C(n,p,\Omega)\, |x-y|^{1-\frac np} \qquad \forall\ x, y\in\partial\Omega.
\end{equation}
\end{lemma}

\begin{proof}
By Lemma \ref{prop-min} there exist functions $u, v \in W^{1,p}(\Omega)$ such that 
\begin{equation} \label{u-v-norm}
 \| u\|_{L^p(\Omega)}^p = \| v\|_{L^p(\Omega)}^p =1, 
\end{equation}
and
\begin{equation} \label{u-v}
\lambda_1(x,\Omega) = \|\nabla u\|_{L^p(\Omega)}^p, \ \ u(x)=0, \qquad \lambda_1(y,\Omega) = \|\nabla v\|_{L^p(\Omega)}^p, \ \ v(y)=0\, .
\end{equation}
Using $u(\cdot)-u(y)$ as a test function for $\lambda_1(y,\Omega)$ we obtain 
\begin{align} \label{1-est}
\lambda_1(y,\Omega) &\ \leq \ \frac{\|\nabla u\|_{L^p(\Omega)}^p}{\|u(\cdot)-u(y)\|_{L^p(\Omega)}^p} \, .
\end{align}
By the Taylor expansion
$$
|u(x)-u(y)|^p \ \geq \ |u(x)|^p -p \, |u(y)|\, |u(x)|^{p-1}
$$
holds for all $x\in\Omega$. Hence by using the H\"older inequality and \eqref{u-v-norm} we obtain
$$
\int_\Omega |u(x)-u(y)|^p\, dx \ \geq \ 1 -p\, |u(y)| \int_\Omega |u(x)|^{p-1}\, dx \ \geq \ 1-p\, |u(y)| \, |\Omega|^{\frac 1p}.
$$
Inserting this lower bound into \eqref{1-est} and taking into account \eqref{u-v} gives 
\begin{align} \label{2-est}
\lambda_1(y,\Omega) &\ \leq \ \frac{\lambda_1(x,\Omega)}{1- p\, |u(y)| \, |\Omega|^{\frac 1p}} \, .
\end{align}
Now, equation\eqref{u-v} and the Morrey inequality yields
\begin{equation} \label{morrey}
|u(y)| \ \leq c\, |x-y|^{1-\frac np}, \qquad |v(x)| \ \leq c\, |x-y|^{1-\frac np},
\end{equation}
where $c$ depends only on $n,p$ and $\Omega$. Since $\frac{1}{1-t} \leq 1 +2t$ for $0\leq t\leq 1/2$, we conclude from \eqref{2-est} that for $|x-y|$ small enough
\begin{align*} \label{upperb-1}
\lambda_1(y,\Omega) &\ \leq \ \lambda_1(x,\Omega) +2 p c\, |\Omega|^{\frac 1p} \,  |x-y|^{1-\frac np} .
\end{align*}
In the same way, using  $v(\cdot)-v(x)$ as a test function for $\lambda_1(x,\Omega)$ we get 
$$
\lambda_1(x,\Omega) \ \leq \ \lambda_1(y,\Omega) +2 p c\, |\Omega|^{\frac 1p} \,  |x-y|^{1-\frac np} .
$$
This proves the claim. 
\end{proof}

\section{\bf Estimates on $\lambda(m,\Omega)$ and $\Lambda(m,\Omega)$}
\label{sec-estimates}
\noindent The aim of this section is to establish lower bounds for $\lambda(m,\Omega)$ and $\Lambda(m,\Omega)$ in terms of $m$ and $|\Omega|$.  
 
\subsection{A lower bound for $\Lambda(m,\Omega)$}
\begin{prop} 
\label{below}
For any $p>1$ and $m>0$ it holds that
\begin{equation}
\label{belsup}
\Lambda(m,\Omega)\ge \frac{m\, \Lambda_{1}^{D}(\Omega)}{\left[\left(|\Omega|\, \Lambda_{1}^{D}(\Omega)\right)^{1/(p-1)}+m^{1/(p-1)}\right]^{p-1}}.
\end{equation}
\end{prop}
\begin{proof}
For the sake of simplicity, we denote again $\Lambda_1^{D}=\Lambda_{1}^{D}(\Omega)$.
From the proof of Theorem \ref{thm3.1} it follows that
\[
\Lambda(m,\Omega) = \ell_{1}(\sigma_{m},\Omega)
= \inf_{u\in \mathcal F}\mathcal Q[\sigma_{m},u]
\]
where
\[
\mathcal F= \{u\in W^{1,p}(\Omega),\,\|u\|_{p}=1, u=k\text{ on }\de\Omega,\,0\le k \}.
 \]
If $u\in \mathcal F$, recalling the variational characterisation of $\Lambda_1^{D}$ we have that
\begin{align}
\notag
\mathcal Q[\sigma_{m},u] &= \int_{\Omega}|\nabla u|^{p}dx + k^{p}m
= \int_{\Omega}|\nabla (u-k)|^{p}dx + k^{p}m \\
 \notag &\ge \Lambda_1^{D}\int_{\Omega}|u-k|^{p}dx +k^{p}m \\
&\label{eqproof}
\ge \Lambda_1^{D}\left|1-k|\Omega|^{\frac1p}\right|^{p}+k^{p}m,
\end{align}
where the last line follows by the Minkowski inequality and the condition $\|u\|_{L^{p}(\Omega)}=1$. By minimising \eqref{eqproof} with respect to $k$ we have that
\[
\mathcal{Q}[\sigma_{m},u] \ge
\frac{m\Lambda_1^{D}}{\left[\left(|\Omega|\, \Lambda_1^{D}\right)^{1/(p-1)}+m^{1/p-1}\right]^{p-1}},
\]
and the thesis follows.
\end{proof}

\subsection{A lower bound for $\lambda(m,\Omega)$}

\begin{prop} \label{pro-inf-lb}
Let $p>n$. Then
\begin{equation}
\label{ineq:inflow}
\lambda(m,\Omega)\ge\frac{m\lambda_1(\Omega)}{\left[\left(|\Omega|\, \lambda_1(\Omega)\right)^{1/(p-1)}+m^{1/(p-1)}\right]^{p-1}},
\end{equation}
where $\lambda_1(\Omega)$ is defined in \eqref{lambda0}.
\end{prop}
\begin{proof}
By Theorem \ref{carmin} we have
\begin{equation}
\label{step1}
\lambda(m,\Omega) = \ell_1(\mu_m,\Omega) = \int_{\Omega}|\nabla \bar u_m|^{p} dx+ m\,  \bar u_m^p(x_m),
\end{equation}
where $\|\bar u_m\|_{L^p(\Omega)}=1.$ Then arguing as in the proof of Proposition \ref{below} and recalling \eqref{dirpunto}, by \eqref{step1} we have
\begin{equation*}
\lambda(m,\Omega)=\ell_1( \mu_m,\Omega) \ge \frac{m\lambda_1(x_m;\Omega)}{\left[\left(|\Omega|\, \lambda_1(x_m;\Omega)\right)^{1/(p-1)}+m^{1/p-1}\right]^{p-1}}.
\end{equation*}
Maximizing the right hand side with respect to $x_m$ gives the claim.
\end{proof}

\begin{rem}
Clearly, by Proposition \ref{inf0}, the inequality \eqref{ineq:inflow} is trivial if $p\le n$, being all the quantities involved equal to zero.
\end{rem}

\subsection{Upper bounds and limiting behavior}

In view of \eqref{lambda-upperb1}, for any $x\in\de\Omega$ and any $m>0$ it holds
\[
\lambda(m,\Omega)\le \lambda_{1}(\Omega) .
\]
This in combination with \eqref{uptriv} gives
\begin{equation} 
\lambda(m,\Omega)\le \min\left\{\lambda_1(\Omega),\frac{m}{|\Omega|}\right\}, \qquad \Lambda(m,\Omega)\le \min\left\{\Lambda_1^D(\Omega),\frac{m}{|\Omega|}\right\}
\end{equation}  
An immediate consequence of these estimates and Propositions \ref{below} and \ref{pro-inf-lb} is the following

\begin{cor}
We have 
\begin{align*}
\lim_{m\to0} \lambda(m,\Omega)& =
\lim_{m\to0} \Lambda(m,\Omega)=0,  \\
\lim_{m\to+\infty} \Lambda(m,\Omega)& =\Lambda_{1}^{D}(\Omega), \\
\lim_{m\to+\infty} \lambda(m,\Omega)& =\lambda_1(\Omega) \qquad \text{\rm if} \ p>n. 
\end{align*}
\end{cor}

\medskip

\begin{rem} We observe that it is possible to study the behavior of $\Lambda(m,\Omega)=\Lambda(m,\Omega;p)$ as $p\to 1$. It is well known that as $p\to 1$, the first Dirichlet eigenvalue $\Lambda_{1}^{D}(\Omega)=\Lambda_{1}^{D}(\Omega;p)$ converges to the Cheeger constant $h(\Omega)$, namely
\[
	h(\Omega)= \inf_{E\subset \Omega}\frac{P(E)}{|E|}
\]
(see for example the survey paper \cite{parini} and the references therein).
Hence, the bounds \eqref{belsup} and \eqref{uptriv} give
\begin{equation*}
\lim_{p\to 1} \Lambda(m,\Omega;p) = \min\left\{\frac{m}{|\Omega|}, h(\Omega) \right\}.
\end{equation*}
\end{rem}
\begin{rem}
We finally recall that if $\sigma(x)=\sigma$ is a positive constant, and $\Omega$ is a bounded convex set, it is possible to obtain a lower bound of $\ell_1(\sigma,\Omega)$ in terms of the inradius $R_\Omega$ of $\Omega$. Indeed, applying \cite[Proposition 3.1]{pota} (see also \cite[Theorem 4.5]{kov} for $p=2$) then
\[
\ell_{1}(\sigma,\Omega) \ge \left(\frac{p-1}{p}\right)^{p} 
\frac{\sigma}{R_{\Omega}\left(1+\sigma^{\frac{1}{p-1}}R_{\Omega}\right)^{p-1}}
\]
\end{rem}

{\bf Acknowledgement.} This work has been partially supported by Gruppo Nazionale per Analisi Matematica, la Probabilit\`a e le loro Applicazioni (GNAMPA) of the Istituto Nazionale di Alta Matematica (INdAM).
The support of FIRB 2013 project ``Geometrical and qualitative aspects of PDE's'', and of MIUR-PRIN2010-11 grant for the project ``Calcolo delle variazioni'' (H.~K.), is also gratefully acknowledged.

\end{document}